\documentclass[11pt,twoside,a4paper]{article}
\usepackage{amsmath,amssymb,amsthm}

\usepackage{graphicx,psfrag}

\newtheorem{Thm}{Theorem}[section]
\newtheorem{Lem}[Thm]{Lemma}

\theoremstyle{definition}

\theoremstyle{remark}

\newcommand{\R}{\mathbb{R}}

\newcommand{\ga}{\gamma}

\newcommand{\la}{\lambda}
\newcommand{\La}{\Lambda}
\renewcommand{\phi}{\varphi}

\newcommand{\Ric}{\operatorname{Ric}}
\newcommand{\tr}{\operatorname{tr}}
\newcommand{\Id}{\operatorname{Id}}

\newcommand{\s}{\operatorname {sin}}
\renewcommand{\cos}{\operatorname{cos}}

\begin{document}

\title {On 3-dimensional Asymptotically Harmonic Manifolds}     
\author{Viktor Schroeder\footnote{Supported by Swiss National
Science Foundation} \ \& Hemangi Shah\footnote{The author thanks 
Forschungsinstitut f{\"u}r Mathematik, ETH Z{\"u}rich for its hospitality
and support}}
\maketitle

\begin{abstract}
\noindent
Let $(M,g)$ be a complete, simply connected Riemannian manifold of
dimension $3$ without conjugate points. We 
show that $M$ is a
hyperbolic manifold of constant sectional curvature $\frac{-h^2}{4}$, provided
$M$ is asymptotically harmonic of constant $h > 0$. 
\end{abstract}
\section{Introduction}
\indent 
Let $(M,g)$ be a complete, simply connected Riemannian manifold 
without conjugate points.  Let $SM$ be the unit
tangent bundle of $M$. For $v \in SM$, let $\gamma_{v}$ be the geodesic 
with ${\gamma'_{v}}(0) = v$ and 
$ b_{v,t} (x) =  \lim_{t\to \infty} (d (x,\gamma_v(t)) - t)$
the corresponding {\it Busemann function} for $\gamma_{v}$.
The level sets ${b_v}^{-1}(t)$
are called {\it horospheres}.\\ 
\indent
A complete, simply connected Riemannian manifold without conjugate 
points is called {\it asymptotically harmonic} if the mean curvature
of its horospheres is a universal constant,
that is if its Busemann functions satisfy 
$\Delta b_v \equiv h,\; \forall v \in SM$, where $h$
is a nonnegative constant. Then $b_v$ is a smooth function on $M$ for all
 $v$ and all horospheres of $M$ are smooth, simply connected 
hypersurfaces in $M$ with constant mean curvature $h$.\\ 
For example, every simply connected, complete harmonic manifold without 
conjugate points is asymptotically harmonic.\\
\indent
For more details on this subject we refer to the 
discussion and to the references in \cite{H}.
In \cite{H} the following result was proved:\\
Let $M$ be a Hadamard manifold of dimension $3$ whose sectional curvatures 
are bounded from above by a negative constant (i.e. $K \le - a^2$ for some
$a \ne 0$) and
whose curvature tensor 
satisfies $\| \nabla R\| \le C$ for a suitable constant $C$. 
If $M$ is asymptotically harmonic, then $M$ is symmetric and hence of constant
sectional curvature.\\
\indent
We prove this result without any hypothesis 
on the curvature tensor.\\
\begin{Thm}\label{thm:mainthm}
 Let $(M,g)$ be a complete, simply connected Riemannian manifold of
dimension $3$ without conjugate points. If $M$ is asymptotically harmonic of
constant $h > 0$, then $M$ is a manifold of constant sectional 
curvature $\frac{-h^2}{4}$. 
\end{Thm} 

\section{Proof of the Theorem}

The first part of the proof (Lemma \ref{lem:detconstant} to Lemma \ref{lem:v})
is a modification of the results in \cite{H}.
Therefore we recall some notations which were already used in that paper.
Our general assumption is that $M$ is 3-dimensional, has no conjugate points and
is asyptotically harmonic with constant $h > 0$.
For $v \in SM$ and $x \in v^{\perp},$ let
$$ u^{+}(v)(x) = \nabla_x \nabla b_{-v} \ \ \ \mbox{and}\ \ \ 
u^{-}(v)(x) = - \nabla_x \nabla b_{v}.$$ 
Thus $u^{\pm}(v) \in \mbox{End} \;(v^{\perp}).$
With $\la_1(v),\la_2(v)$ we denote the eigenvalues of $u^+(v)$.
The endomorphism $u^{\pm}$ satisfy the Riccati equation along the orbits 
of the geodesic flow ${\phi}^{t} : SM \rightarrow SM$.\\
Thus if
$u^{\pm}(t):=u^{\pm}(\phi^tv)$ and 
$R(t):=R(\cdot,\gamma_v'(t))\gamma_v'(t)\in\mbox{End}(\gamma_v'(t)^{\perp}),$
then
 $$(u^{\pm})' + (u^{\pm})^2 + R = 0.$$
We define
$V(v)=u^+(v)-u^-(v)$ and correspondingly
$V(t) = V(\phi^t(v))$ along $\gamma_v(t)$.
We also define $X(v)=\frac{-1}{2}(u^+(v)-u^-(v))$ and
$X(t)=X(\phi^t(v))$. Then the Riccati equation for $u^{\pm}(t)$ yields 
\begin{equation} \label{eq:XV}
XV + VX = (u^-)^2 - (u^+)^2 = (u^+)' -  (u^-)' = V'. 
\end{equation}

\begin{Lem}\label{lem:detconstant}
For fixed $v\in SM$ the map
$t\mapsto\det V(\phi^tv)$ is constant.
\end{Lem}
\begin{proof}
Assume that $V(t)$ is invertible, then
$$ \frac{d}{dt} \log \det V(t) = {\mbox tr}\; V'(t) V^{-1}(t) =
   \mbox{tr}\;(XV + VX) V^{-1}(t) = 2\; \mbox{tr} X = 0.$$  
The last step follows as $M$ is asymptotically harmonic.
Thus as long as $\det V(t) \ne 0$, it is constant along $\gamma_v$.
Therefore $\det V(t)$ is constant along $\gamma_v$ in any case.
\end{proof}

\begin{Lem}\label{lem:scal}
Let $v \in SM$ be such that $V(v) = \mu\Id$, for
some $\mu\in \R$, then 
$R(t) = \frac{-h^{2}}{4}\Id$, $\forall t$.   
\end{Lem}
\begin{proof} Note that if  $V(v) = \mu\Id$, then  
$V(\ga_{v}'(t)) = h\Id$ for all $t$, as $\tr V \equiv 2h$ and by Lemma \ref{lem:detconstant} 
the determinant of $V$ 
is constant along $\ga_{v}(t)$. Now by equation (\ref{eq:XV}) 
$V' = XV + VX. $ Hence, along $\ga_{v}$, $V'(t)  \equiv 0.$
Thus $2h X = 0$ and since we assume $h>0$ we have $X = 0$ along $\ga_{v}$. Therefore, 
$u^+(t) = - u^-(t)$. But from the definition of $V$, $ u^+(t) \equiv
\frac{h}{2}\Id$ i.e $u^+$ is a scalar operator. By the Riccati equation
$(u^+(t))^2 + R(t) = 0,$ i.e. $R(t) = \frac{-h^2}{4}\Id$. 
\end{proof}

\begin{Lem}\label{lem:v}
For every point $p \in M$ there exists $v \in S_{p}M$
such that $R(x, v)v = \frac{-h^{2}}{4}\;x,\; \forall x \in v^{\perp}.$
In particular, $\Ric(v, v) =  \frac{-h^{2}}{2}.$  
\end{Lem}

\begin{proof}
Since $TS^2$ is nontrivial, an easy topological argument shows,
that for every $p \in M$ there exists $v \in S_pM$ such that
the two eigenvalues of $V(v)$ coincide.
Thus $V(v) = \mu\Id.$
The result now follows from Lemma \ref{lem:scal}

\end{proof}

\begin{Lem} \label{lem:estimatericci}
 For all $v \in SM$ we have $\Ric(v, v) \leq  \frac{-h^{2}}{2}.$ 
 \end{Lem} 
\begin{proof} The Riccati equation for $t\mapsto u^{+}(t)$ implies
 $(u^{+})' + (u^{+})^2 + R = 0.$ Hence, $\tr(u^{+})^2 + \tr R = 0.$ 
 Thus, $\Ric(v, v) = -({{\la}_1}^{2}(v) +{{\la}_2}^{2}(v)).$ 
By hypothesis ${\la}_1(v) + {\la}_2(v) = h,$ hence
 ${{\la}_1}^{2}(v) +{{\la}_2}^{2}(v) \geq \frac{h^{2}}{2}.$
 Consequently, $\Ric(v, v) \leq  \frac{-h^{2}}{2}.$  
\end{proof}

\begin{Lem}\label{lem:curvature}
The sectional curvature $K$ of $M$ satisfies
$K \le -\frac{h^2}{4}$.
\end{Lem}

\begin{proof}
Let $p \in M$, and let $v$ be the vector in Lemma \ref{lem:v}.
Take $e_1 = v,$ and let $e_2$ and $e_3$ be unit vectors 
orthogonal to $e_1$ so that $\{ e_1, e_2, e_3 \}$ forms an orthonormal
basis of $T_pM$. Then $\{e_1 \wedge e_2,\; e_1 \wedge e_3,\;  e_2 \wedge e_3\}$ forms an 
orthonormal basis of ${\La}^2T_pM$. 
We want to show that the curvature operator,
considered as map
$R : {\La}^2 T_pM \rightarrow {\La}^2 T_pM$,
$\left<R(X \wedge Y), V \wedge W\right> =\left<R(X,Y)W, V\right>$ 
is diagonal in this basis.\\
From Lemma \ref{lem:v} we see
$R(e_2, e_1)e_1 = \frac{-h^{2}}{4}\; e_2,$\; 
$R(e_3, e_1)e_1 = \frac{-h^{2}}{4}\; e_3.$ 
Thus  
$K(e_1, e_2) = K(e_1, e_3) = \frac{-h^{2}}{4}$\; and 
$K(e_2, e_3) \leq  \frac{-h^{2}}{4}$\; as\; 
$\Ric(e_3,e_3) \leq \frac{-h^{2}}{2},$  
where $K(v,w)$ denotes the sectional curvature of the plane
 spanned by $v$ and $w$. 
 We will prove below that
\begin{eqnarray} \label{eq:12}
 \left<R(e_1, e_3)e_3, e_2 \right> = 0 \ \ \mbox{and} \ 
\left<R(e_1, e_2)e_2, e_3 \right> = 0.
\end{eqnarray}
Assuming this for a moment, it follows that
$R(e_1 \wedge e_3) \perp \mbox{span} \{e_1 \wedge e_2,\; e_2 \wedge e_3\}$
and
$R(e_1 \wedge e_2) \perp \mbox{span} \{e_1 \wedge e_3, e_2 \wedge e_3\}.$    
Hence, 
$$R(e_1 \wedge e_2) = \frac{-h^{2}}{4}\; e_1 \wedge e_2 \;\mbox{and}\;
  R(e_1 \wedge e_3) = \frac{-h^{2}}{4}\; e_1 \wedge e_3.$$   
Since $e_1 \wedge e_2$ and $e_1 \wedge e_3$ are eigenvectors of $R$, also
$e_2 \wedge e_3$ is an eigenvector and we obtain
$$ R(e_2 \wedge e_3) =
 K(e_2,e_3)\; e_2 \wedge e_3.$$ 
Thus the curvature operator is diagonal in the basis
$\{e_1 \wedge e_2,\; e_1 \wedge e_3,\;  e_2 \wedge e_3\}$
and all eigenvalues are $\le \frac{-h^{2}}{4},$ 
which proves the result. 

It remains to show (\ref{eq:12}). 
Consider for $t \in (-\varepsilon,\varepsilon)$ the vectors
$v_t = \cos t e_1   + \sin t e_2.$ 
Then,
\begin{eqnarray*} 
 f(t) := \Ric(v_t,v_t) = K(v_t, e_3) + 
K(v_t, - e_1 \s t + e_2 {\cos} t)\\ 
= K(e_1, e_2) + {\s}^{2}t \; K(e_2,e_3) + {\cos}^{2}t \; K(e_1,e_3)
  +{\s}2t \left<R(e_1, e_3)e_3, e_2 \right>.
\end{eqnarray*}    
By Lemma \ref{lem:estimatericci}
 $f(0) = \Ric(v,v) =\frac{-h^2}{2} $ 
is maximal  and
hence $f'(0) = 0$. This implies the first equation in
(\ref{eq:12}).
If we replace $e_2$ by $e_3$ in the above computation 
we obtain the second equation.

\end{proof}

Finally we come to the 

\begin{proof} (of Theorem \ref{thm:mainthm}) \\
Lemma \ref{lem:curvature} implies that $K_M \leq \frac{-h^{2}}{4}.$
By standard comparison geometry we obtain
$\la_1(v) \geq \frac{h}{2}$ and $\la_2(v) \geq \frac{h}{2}$.
Now $\la_1 + \la_2 = h$ implies that
$\la_1 = \la_2 = \frac{h}{2}.$  
Hence, $u^{+}(v)$ is a scalar operator
and therefore $R(x,v)v = \frac{-h^{2}}{4}\; x,\; \forall v$ and 
$\forall x \in {v}^{\perp}.$ Thus,\; $K_M \equiv \frac{-h^{2}}{4}$.
\end{proof}

\section{Final Remark} 
We expect that the result holds also in the case $h=0$, i.e. 
if all horospheres
are minimal. Our argument, however, uses $h >0$ essentially in the proof
of Lemma \ref{lem:scal}

Institut f{\"u}r Mathematik, Universit{\"a}t Z{\"u}rich, 
Winterthurerstrasse 190,\\
\indent
CH-8057 Z{\"u}rich.\\
\indent
Email :  vschroed@math.unizh.ch\\\\
\indent
School of Mathematics, Tata Institute of Fundamental Research,\\
\indent
Dr. Homi Bhabha Road, Mumbai - 400005, India.\\
\indent
Email :  hema@math.tifr.res.in

\end{document}